\documentclass[a4paper, 12pt]{amsart}
\usepackage[dvips]{graphicx}
\usepackage[all]{xy}

\newtheorem{thm}{Theorem}[section]
\newtheorem{prop}[thm]{Proposition}

\newtheorem{definition}[thm]{Definition}
\newtheorem{corollary}[thm]{Corollary}

\newtheorem{theorem}{Theorem}[section]
\newtheorem{lemma}[theorem]{Lemma}

\newtheorem{remark}{\bf Remarks}[section]

\numberwithin{equation}{section} 

\newcommand{\bea} {\begin{eqnarray*}}
\newcommand{\beq} {\begin{equation}}
\newcommand{\bey} {\begin{eqnarray}}
\newcommand{\eea} {\end{eqnarray*}}
\newcommand{\eeq} {\end{equation}}
\newcommand{\eey} {\end{eqnarray}}

\begin{document}
\title[One-parameter Lefschetz class of homotopies on torus]
{One-parameter Lefschetz class of homotopies on torus}
\author{Weslem L. Silva}
\address[Weslem Liberato Silva]{Departamento de Ci\^encias Exatas e Tecnol\'ogicas (DCET), Universidade Estadual de Santa Cruz (UESC), Ilh\'eus - BA, 2015}
\email{wlsilva@uesc.br}
\begin{abstract}
The main result this paper states that if $F: T \times I \to T$ is a homotopy on torus then the one-parameter Lefschetz class $L(F)$ of $F$ is given by $L(F) = \pm N(F)\alpha$, where $N(F)$ is the one-parameter Nielsen number of $F$ and $\alpha$ is one of the two generators in 
$H_{1}(\pi_{1}(T),\mathbb{Z})$.
\end{abstract}
\date{\today}
\maketitle
\bibliographystyle{amsplain}

\section{Introduction}

Let $F: T \times I \to T$ be a homotopy on torus and $G = \pi_{1}(T,x_{0})$. R.Geoghegan and A. Nicas in \cite{G-N-94} developed an one-parameter theory and defined the one-parameter trace $R(F)$ of $F$. The element $R(F)$ is a 1-chain in $HH_{1}(\mathbb{Z}G,(\mathbb{Z}G)^{\phi}))$, where the structure of the bimodule $(\mathbb{Z}G)^{\phi})$ is given in section 2. This 1-chain gives information about the fixed points of $F$, that is, using $R(F)$ is possible to define the one-parameter Nielsen number $N(F)$ of $F$ and the one-parameter Lefschetz class $L(F)$ of $F$. N(F) is the number of non-zero C-components in $R(F)$ and $L(F)$ is the image of $R(F)$ in $H_{1}(G)$ by homomorfism $\bar{j_{C}}: H_{1}(Z(g_{C})) \to H_{1}(G)$, induced by inclusion $j_{C}: Z(g_{C}) \to G$, where $Z(g_{C})$ is the semicentralizer of an element $g_{C}$ which represents the semiconjugacy class $C$. The precise definition is given in \cite{G-N-94}.

The main purpose this paper is show that for which homotopy on torus then $L(F) = \pm N(F) \alpha$, where $\alpha$ is on of the two generators in $H_{1}(G)$.    
  
In \cite{B-B-P-T-75} R.B.S.Brooks et al. showed that if $f: X \to X$ is any map on a k-dimensional torus $X$ then $N(f) = |L(f)|$, where 
$N(f)$ is the Nielsen number and $L(f)$ the Lefschetz number of $f$. In some sense our result is a version of this result for  
one-parameter case when $k=2$. 

This paper is organized into five sections, besides this one. In Section 2 contain a review of one-parameter fixed point theory. 
In section 3 we presented some results of semiconjugacy classes on torus. In Section 4 have the proof of the main result which is the 
Theorem \ref{maintheorem}. 


\section{One-parameter Fixed Point Theory}

Let $X$ be a finite connected CW complex and $F: X \times I \to X$ a cellular homotopy.
We consider  $I = [0, 1]$ with the usual CW structure and orientation of cells, 
and $X \times I$ with the product CW structure, where its cells are given the product orientation.
Pick a basepoint $(v,0) \in X \times I$, and a basepath $\tau$ in $X$ from $v$ to $F(v,0)$. 
We identify $\pi_{1}(X \times I, (v,0)) \equiv G$ with $\pi_{1}(X,v) $ via the isomorphism 
induced by projection $p: X \times I \to X$. We write $\phi: G \to G$ for the homomorphism;
$$ \pi_{1}(X \times I, (v,0)) \stackrel{F_{\#}}{\to} \pi_{1}(X, F(v,0)) \stackrel{c_{\tau}}{\to} \pi_{1}(X, v) $$

We choose a lift $\tilde{E}$ in the universal cover, $\tilde{X}$, of $X$ for each cell $E$ and 
we orient $\tilde{E}$ compatibly with $E$. 
Let $\tilde{\tau}$ be the lift of the basepath $\tau$ which starts at in the basepoint $\tilde{v} \in \tilde{X}$ 
and $\tilde{F}: \tilde{X} \times I \to \tilde{X}$ the unique lift of $F$ satisfying $\tilde{F}(\tilde{v},0) = \tilde{\tau}(1)$.
We can regard $C_{\ast}(\tilde{X})$ as a right $\mathbb{Z}G$ chain complex as follows: if $\omega$ is a loop 
at $v$ which lifts to a path $\tilde{\omega}$ starting at $\tilde{v}$ then $\tilde{E}[\omega]^{-1} = h_{[w]}(\tilde{E})$, 
where $h_{[\omega]} $ is the covering transformation sending $\tilde{v}$ to $\tilde{\omega}(1)$.

The homotopy $\tilde{F}$ induces a chain homotopy $\tilde{D_{k}}: C_{k}(\tilde{X}) \to C_{k+1}(\tilde{X})$ 
given by $\tilde{D_{k}}(\tilde{E}) = (-1)^{k+1}F_{k}(\tilde{E} \times I) \in C_{k+1}(\tilde{X})$, for each cell $\tilde{E} \in \tilde{X}$. 
This chain homotopy satisfies; $\tilde{D}(\tilde{E}g) = \tilde{D}(\tilde{E}) \phi(g)$ and the boundary operator 
$\tilde{\partial_{k}}: C_{k}(\tilde{X}) \to C_{k-1}(\tilde{X})$ satisfies; $\tilde{\partial}(\tilde{E}g)= \tilde{\partial}(\tilde{E})g$.

Define endomorphism of, $\oplus_{k} C_{k}(\tilde{X})$, by $\tilde{D_{\ast}} = \oplus_{k} (-1)^{k+1} \tilde{D_{k}}$,
$ \tilde{\partial_{\ast}} = \oplus_{k} \tilde{\partial_{k}}$, $\tilde{F_{0 \ast}} = \oplus_{k} (-1)^{k} \tilde{F_{0 k}} $ 
and $\tilde{F_{1 \ast}} = \oplus_{k} (-1)^{k} \tilde{F_{1 k}} $.
We consider  trace$(\tilde{\partial_{\ast}} \otimes \tilde{D_{\ast}}) \in HH_{1}(\mathbb{Z}G, (\mathbb{Z}G)^{\phi})$. 
This is a Hochschild 1-chain whose boundary is:
trace$(\tilde{D_{\ast}}\phi(\tilde{\partial_{\ast}}) - \tilde{\partial_{\ast}} \tilde{D_{\ast}}) .  $
We denote by $G_{\phi}(\partial(F))$ the subset of $G_{\phi}$ consisting of semiconjugacy classes associated to fixed 
points of $F_{0}$ or $F_{1}$.

\begin{definition}
The  one-parameter trace of homotopy $F$ is:
$$R(F) \equiv T_{1}(\tilde{\partial_{\ast}} \otimes \tilde{D_{\ast}}; G_{\phi}(\partial(F))) \in 
\bigoplus_{C \in G_{\phi} - G_{\phi}(\partial(F))} HH_{1}(\mathbb{Z}G, (\mathbb{Z}G)^{\phi})_{C}  $$
$$ \cong \bigoplus_{C \in G_{\phi} - G_{\phi}(\partial(F))} H_{1}(Z(g_{C})). $$
\end{definition}

\begin{definition} The $C-$component of $R(F)$ is denoted by $i(F,C) \in $ \break $ {HH_{1}(\mathbb{Z}G, (\mathbb{Z}G)^{\phi})}_{C}.$
We call it the  fixed point index of $F$ corresponding to semiconjugacy class $C \in G_{\phi}$. A fixed point index $i(F,C)$ of 
$F$ is zero if the all cycle in $i(F,C)$ is homologous to zero. 
\end{definition}

\begin{definition} Given a cellular homotopy $F: X \times I \to X$ the  one-parameter Nielsen number, $N(F)$, of $F$ 
is the number of nonzero fixed point indices. 
\end{definition}
  
\begin{definition} The  one-parameter Lefschetz class, $L(F)$, of $F$ is defined by;  
$$L(F) = \displaystyle \sum_{C \in G_{\phi} - G_{\phi}(\partial F)} j_{C}(i(F,C)) $$ where 
$j_{C}: H_{1}(Z(g_{C})) \to H_{1}(G)$ is induced by the inclusion $Z(g_{C}) \subset G$.
\end{definition}

From \cite{G-N-94} we have the following theorems. 

\begin{theorem}[Invariance] \label{invariance}
Let $F,G: X \times I \to X$ be cellular; if $F$ is homotopic to $G$ relative to $X\times \{0,1\}$ then $R(F) = R(G)$.
\end{theorem}

\begin{theorem}[one-parameter Lefschetz fixed point theorem]
If $L(F) \neq 0$ then every map homotopic to $F$ relative to $X \times \{0,1\}$ has a fixed point 
not in the same fixed point class as any fixed point in $X \times \{0,1\}$. In particular, if 
$F_{0}$ and $F_{1}$ are fixed point free, every map homotopic to $F$ relative to $X \times \{0,1\}$ has 
a fixed point.   
\end{theorem}   

\begin{theorem}[one-parameter Nielsen fixed point theorem] \label{nielsen-number}
Every map homotopic to $F$ relative to $X \times \{0,1\}$ has at least $N(F)$ fixed point classes 
other than the fixed point classes which meet $X \times \{0,1\}$. In particular, if $F_{0}$ and $F_{1}$ 
are fixed point free maps, then every map homotopic to $F$ relative to $X \times \{0,1\}$ has 
at least $N(F)$ path components.
\end{theorem}

For a complete description of the one-parameter fixed point theory see \cite{G-N-94}.

\section{Semiconjugacy classes on torus}

In this subsection we describe some results about the semiconjugacy classes in the torus related to 
a homotopy $F: T \times I \to T$. We will consider the homomorphism $\phi = c_{\tau} \circ F_{\#}$ given above. 

We take $w=[(0,0)] \in T$ and $G = \pi_{1}(T,w) = \{ u, v| uvu^{-1}v^{-1}=1 \}$, where 
$u \equiv a$ and $v \equiv b$. Thus, given the homomorphism $\phi: G \to G$ we have 
$\phi(u) = u^{b_{1}} v^{b_{2}}$ and $\phi(v) = u^{b_{3}} v^{b_{4}}$. Therefore, 
$\phi(u^{m}v^{n}) = u^{mb_{1}+nb_{3}} v^{mb_{2}+nb_{4}}$, for all $m,n \in \mathbb{Z}$.
We denote this homomorphism by the matrix:
$$ [\phi] = \left ( 
\begin{array}{cc}
  b_{1} & b_{3} \\
  b_{2} & b_{4} \\
\end{array} \right )
$$

\begin{prop} \label{prop-conjugacy-classes}
Two elements  $ g_{1} = u^{m_{1}}v^{n_{1}}$ and  $ g_{2} = u^{m_{2}}v^{n_{2}} $ in $G$ belong to the same semiconjugacy class,
 if and only if there are integers $m,n$ satisfying the following equations:
$$ \left \{
\begin{array}{c}
 m(b_{1}-1)+nb_{3} = m_{2} - m_{1} \\
 mb_{2}+n(b_{4}-1) = n_{2} - n_{1} \\
 \end{array} \right. 
$$ 
\end{prop}
\begin{proof}
If there is $g = u^{m} v^{n} \in G$ satisfying $g_{1} =g g_{2} \phi(g)^{-1}$ then we obtain the equation of the proposition. 
The other direction is analogous.
\qed
\end{proof}

We take the isomorphism $\Theta : G \to \mathbb{Z} \times \mathbb{Z}$ such that $\Theta(u^{m}v^{n}) = (m, n)$.
By Proposition \ref{prop-conjugacy-classes} two elements $ g_{1} = u^{m_{1}}v^{n_{1}}$ and  $ g_{2} = u^{m_{2}}v^{n_{2}} $ in $G$
 belong to the same semiconjugacy class, if and only if 
there is $z \in \mathbb{Z} \times \mathbb{Z}$ satisfying:  
$([\phi]-I)z = \Theta(g_{2}g_{1}^{-1})$, where 
$I$ is the identity matrix. If determinant of the matrix $([\phi]-I) $ is zero then will have an infinite amount of elements 
in a semiconjugacy class.

\begin{corollary} \label{corol-2} For each $g \in G$ the semicentralizer $Z(g)$ is isomorphic to the kernel of $[\phi]-I$.
\end{corollary}

\begin{lemma} \label{lemma-1}
The 1-chain,  $u^{k} v^{l} \otimes u^{m} v^{n}$, is a cycle if and only if the element $(k,l) \in \mathbb{Z} \times \mathbb{Z}$ 
belongs to the kernel of $[\phi]-I$.
\end{lemma}
\begin{proof} 
If $u^{k} v^{l} \otimes u^{m} v^{n} $ is a cycle, then 
$0 = d_{1}( u^{k} v^{l} \otimes u^{m} v^{n} ) = $
$ u^{m} v^{n} \phi(u^{k} v^{l}) - u^{k} v^{l} u^{m} v^{n} = $ 
$ u^{m} v^{n} u^{kb_{1}+lb_{3}} v^{kb_{2}+lb_{4}} - u^{k} v^{l} u^{m} v^{n} = $ 
$u^{m + kb_{1}+lb_{3}} v^{kb_{2}+lb_{4}+n} - $ \break $ u^{k+m} v^{l+n} $. This implies 
$k(b_{1}-1)+lb_{3} = 0$ and $kb_{2}+l(b_{4}-1) = 0$. 
The other direction is analogous. \qed 
\end{proof}

\begin{corollary} \label{corollary-1}
If the matrix of the homomorphism $\phi$ is given by $$ [\phi] = \left ( 
\begin{array}{cc}
  1 & b_{3} \\
  0 & b_{4} \\
\end{array} \right )
$$ with $b_{3} \neq 0$ or $ b_{4} \neq 1$, then the 1-chain, $u^{k} v^{l} \otimes u^{m} v^{n} $, is a cycle if and only if 
$l = 0$.
\end{corollary}

\bigskip

By definition given a 2-chain $ u^{s}v^{t} \otimes u^{k}v^{l} \otimes u^{m}v^{n} \in C_{2}(\mathbb{Z}G, (\mathbb{Z}G)^{\phi})_{C}$ then 
$$\begin{array}{l}
d_{2}(u^{s}v^{t} \otimes u^{k}v^{l} \otimes u^{m}v^{n}) 
 = u^{k}v^{l} \otimes u^{m+sb_{1}+tb_{3}} v^{n+sb_{2}+tb_{4}} -u^{k+s}v^{l+t} \otimes u^{m}v^{n}  \\
 + u^{s}v^{t} \otimes u^{k+m}v^{l+n}. \\
\end{array}$$

\begin{prop} \label{prop-1}
The 1-chain, $u^{k} \otimes u^{m} v^{n} \in C_{1}(\mathbb{Z}G,(\mathbb{Z}G)^{\phi})_{C}$, is homologous to the 1-chain, 
$k u \otimes u^{m+k-1} v^{n} $, for all $k,m,n \in \mathbb{Z}$. 
\end{prop}
\begin{proof} 
Note that for $k= 0$ and $1$ the proposition is true. We suppose that  
for some $s > 0 \in \mathbb{Z}$, the 1-chain $u^{s} \otimes u^{m} v^{n}$  
is homologous to the 1-chain $ s u \otimes u^{m+s-1} v^{n} $, we will write $u^{s} \otimes u^{m} v^{n}$  
$ \sim s u \otimes u^{m+s-1} v^{n} $. Considering to the 2-chain
 $u^{s} \otimes u \otimes u^{m}v^{n} \in C_{2}(\mathbb{Z}G,(\mathbb{Z}G)^{\phi})$ we have  
$$
\begin{array}{lll}
 d_{2}(u^{s} \otimes u \otimes u^{m}v^{n} ) 
 & = &  u \otimes u^{m+s}v^{n} - 
 u^{s+1} \otimes  u^{m}v^{n} + u^{s} \otimes  u^{1+m}v^{n} \\
 & \sim &  u \otimes u^{m+s}v^{n} - 
 u^{s+1} \otimes  u^{m}v^{n} + su \otimes  u^{1+m+s-1}v^{n} \\
 & = & (s+1)u \otimes  u^{m+(s+1)-1}v^{n} - u^{s+1} \otimes  u^{m}v^{n}.  \\
\end{array}
$$
Therefore $(s+1)u \otimes  u^{m+(s+1)-1}v^{n} \sim u^{s+1} \otimes  u^{m}v^{n}.$  Using induction, we obtain the result. 
The case in which $k < 0$ is analogous. \qed
\end{proof}

\bigskip

The proof of following results can be found in \cite{S-14}.

\begin{prop}  \label{prop-2}
In the case $b_{1} =1$ and $b_{2} = 0$ each 1-cycle $\displaystyle \sum^{t}_{i=1} a_{i} u^{k_{i}} v^{l_{i}} \otimes u^{m_{i}} v^{n_{i}} \in C(\mathbb{Z}G, (\mathbb{Z}G)^{\phi}) $ is homologous 
to a 1-cycle the following form: $\displaystyle \sum^{\bar{t}}_{i=1} \bar{a_{i}} u \otimes u^{\bar{m_{i}}} v^{\bar{n_{i}}}$.
\end{prop}

\begin{prop} \label{prop-3}
Each 1-cycle $u \otimes u^{m} v^{n} \in HH_{1}(\mathbb{Z}G,(\mathbb{Z}G)^{\phi})_{C}$ is not trivial, that is, is not homologous to zero.
\end{prop}

\begin{corollary} \label{corollary-2}
Let $\displaystyle \sum_{i=1}^{t} u \otimes u^{m_{i}} v^{n_{i}} \in HH_{1}(\mathbb{Z}G,(\mathbb{Z}G)^{\phi}), 
\, m_{i}, n_{i} \in \mathbb{Z}$ be a cycle. If the cycles $u \otimes u^{m_{i}} v^{n_{i}} $ and $u \otimes u^{m_{j}} v^{n_{j}}$ are 
in different semiconjugacy classes for $i \neq j$, $i,j \in \{1, ... , t \}$, then $\displaystyle \sum_{i=1}^{t} u \otimes u^{m_{i}} v^{n_{i}}$is a nontrivial cycle. Each cycle $u \otimes u^{m_{i}}v^{n_{i}}$ projects to the class $[u]$ that is one of the two generators of $H_{1}(G).$
\end{corollary}

\section{Homotopies on torus}

Let $F: T \times I \to T$ be a homotopy on torus $T$.

\begin{prop} \label{prop-h}
Let $F: T \times I \to T$ be a homotopy. Suppose that $L(F_{t}) = 0$ for each $t \in I$. Then $F$ is homotopic to a homotopy $H$ 
with $H$ transverse the projection $P: T \times I \to T$ such that $Fix({H}_{| T \times \{0,1 \}}) = ptyset$.
\end{prop}
\begin{proof} We can choose a homotopy $F_{0}$ homotopic to $F$ with 
$F_{0}$ transverse the projection $P$. Therefore, $Fix(F_{0})$ is transverse, that is, $Fix(F_{0}) \cap (T \times \{t\})$ is finite. 
Since $L(F_{|T}) = L({F_{0}}_{|T}) = 0$ then for each $\frac{1}{2}> \epsilon > 0 $ we can deform $F_{0}$ to a homotopy $F_{1}$ such that $F_{1}(x,t) = F_{0}(x,t)$ for each $(x,t) \in T \times [\epsilon, 1-\epsilon]$ and $F_{1}$ has no fixed points in $T \times \{0,1 \}$. In fact, take $A:T \times I \times I \to T$ defined by
$$ A((x,y),t,s) =  \left \{  
\begin{array}{lll}
 F_{0}(x,y,0) & if & 0 \leq t \leq s\epsilon \\
 F_{0}(x,y,\frac{1}{1-2s\epsilon}(t-s\epsilon)) & if & s\epsilon \leq t \leq 1-s\epsilon \\
 F_{0}(x,y,1) & if & 1-s\epsilon \leq t \leq 1 \\
\end{array} \right.
$$

Since $L({F_{0}}_{|T}) = 0$, there are two homotopies $H_{1}, H_{2}: T \times I \to T$ such that $H_{1}(x,y,1) = F_{0}(x,y,0)$, 
$H_{2}(x,y,0) = F_{0}(x,y,1)$ and $H_{1}(x,y,0)$, $H_{2}(x,y,1)$ are fixed points free maps.  
Considere the homotopy $B: T \times I \times I \to T$ defined by;
$$ B((x,y),t,s) =  \left \{  
\begin{array}{lll}
 H_{1}(x,y,\frac{t}{\epsilon}s) & if & 0 \leq t \leq \epsilon \\
 F_{0}(x,y,\frac{1}{1-2\epsilon}(t-\epsilon)) & if & \epsilon \leq t \leq 1-\epsilon \\
 H_{2}(x,y,\frac{(t-1+\epsilon)}{\epsilon}s) & if & 1-\epsilon \leq t \leq 1 \\
\end{array} \right.
$$

Thus, taking
$$ J((x,y),t,s) =  \left \{  
\begin{array}{lll}
 A((x,y),t,2s) & if & 0 \leq s \leq \frac{1}{2} \\
 B((x,y),t),2s-1)& if & \frac{1}{2} \leq s \leq 1 \\
\end{array} \right.
$$
we have a homotopy between $F_{0}$ and a map $H$ where $H$ satisfying the hypothesis of the theorem. 
Note that we can choose $\frac{1}{2} > \epsilon > 0 $ such that $Fix(F_{0}) \subset T \times [\epsilon, 1-\epsilon]$ because $Fix(F_{0})$ is contained in $int(T \times I)$. Thus, $Fix(H)$ is transverse.  \qed \end{proof}

\bigskip

\begin{figure}[!htp]
\begin{center}
		\includegraphics[scale=0.28]{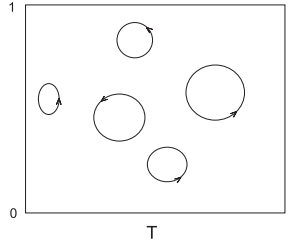}
	\caption{Circles in $Fix(F).$}
		\end{center}
\end{figure}

Let $Fix(F,\partial)$ be the subset of $Fix(F)$ consisting of those circles of fixed points which are not in the same fixed point class as any fixed point of $F_{0}$ or $F_{1}$. From \cite{G-N-94} $Fix(F)$ consists of oriented arcs and circles. 

From Proposition \ref{prop-h} if $F: T \times I \to T$ is a homotopy and $P:T \times I \to T$ the projection then we can choose $F$ such that $Fix(F)$ is transverse the projection $P$. Thus, $Fix(F,\partial)$ is a closed oriented 1-manifold in the interior of $T \times I \times T$. Let $E_{F}$ be space of all paths $\omega(t)$ in $T \times I \times T$ from the graph $\Gamma_{F}= \{(x,t,F(x,t))|(x,t) \in T \times I \}$ of $F$ to the graph $\Gamma_{P} = \{(x,t,x)|(x,t) \in T \times I  \}$ of $P$ with the compact-open topology, that is, maps $\omega:[0,1] \to T \times I \times T$ such that $\omega(0) \in \Gamma(F)$ and $\omega(1) \in \Gamma(P)$. 

Let $C_{1}, ... , C_{k}$ be isolated circles in $Fix(F) \cap int(T \times I)$, oriented by the natural orientations, and $V = \bigcup C_{j}.$ Then $V$ determines a family of circles $V^{'}$ in $E_{F}$ via constant paths, i.e. each oriented isolated circle of fixed points $C: S^{1} \to T \times I$ of $F$ determines an oriented circle $C^{'}: S^{1} \to E$ defined by $con(C(z))$ where $con(C(z)$ is the constant path at $C(z) = (x,t_{0})$, that is, $con(C(z))(t) = (x,t_{0},x)$ for each $t \in [0,1]$. 
Therefore, we can write $\displaystyle \sum i(F,C_{j}).[C_{j}^{'}] \in H_{1}(E_{F})$. Since $C_{j}$ is transverse then $i(F,C_{j})= 1$ for all $j$, see \cite{D-94}. From \cite{G-N-94} we have;

\begin{prop} \label{prop-G}
Since $\pi_{2}(T) = 0$ then there is a isomorphism $\Psi: H_{1}(E_{F}) \to HH_{1}(\mathbb{Z}G, (\mathbb{Z}G)^{\phi})$, where $G = \pi_{1}(T,x_{0})$.
\end{prop}

\begin{remark} From \cite{D-94}, section $IV$, given $F: T \times I \to T$ a homotopy then we can to deform $F$ to a homotopy $G$ such that in each fixed point class of $G$ has an unique circle, and this circle is transverse the projection. 
\end{remark}

Now we are going to proof the main result.

\begin{theorem}[Main Theorem] \label{maintheorem}
If $F: T \times I \to T$  is a homotopy then the one-parameter Lefschetz class $L(F)$ of $F$ satisfies 
$L(F) = \pm N(F)\alpha$ where $\alpha$ is one of the two generators of $H_{1}(\pi_{1}(T),\mathbb{Z})$.
\end{theorem}
\begin{proof}
The proof this theorem will be done in two cases. Case I when $det([\phi]-I)=L(F_{|T}) = 0$ and case II 
when $det([\phi]-I)=L(F_{|T}) \neq 0$.

\bigskip

{\bf Case I}

Let us suppose that the homomorphism $\phi$ is induced by a homotopy $F$ satisfies $det([\phi]-I)=0$. Using the notation above we can suppose which $\phi$ is given by $$ [\phi] = \left ( 
\begin{array}{cc}
  1 & b_{3} \\
  0 & b_{4} \\
\end{array} \right ),
$$
and $[\phi] \neq I \equiv $ (Identity), that is, $b_{1} = 1$ and $b_{2} = 0$, with $b_{3} \neq 0$ or $b_{4}-1 \neq 0$. 
This is done choosing a base $\{v,w\}$ for $T = \mathbb{R}^{2}/\mathbb{Z}^{2}$, where $v$ is a eigenvector of $[\phi]$ associated to $1$.   

Note that if $[\phi] = I$ then $R(F) = 0$ because any $F$ can be deformed to a fixed point free map. For example, take the homotopy 
$F: T \times I \to T$ defined by;
$$
F((x,y),t) = (x+c_{1}t+\epsilon, y+c_{2}t)
$$
with $\epsilon$ any irrational number between $0$ and $1$. We will have $[F_{\#}] = [\phi] = I$, but $F$ is a fixed point free map. Thus 
$R(F) = 0$, which implies $L(F) = N(F) = 0$. Therefore, henceforth we suppose $[\phi] \neq I.$

Since $T$ is a polyhedron then $T$ is a regular CW-complex. Thus, for any cellular decomposition of the torus the entries of matrices of the operators $\tilde{\partial}_{1}$ and $\tilde{\partial}_{2}$ will be composed by elements $0, \pm 1, \pm u,\pm v,$ because the incindence number of a 2-cell in a 1-cell is $\pm 1$ and the the incindence of one 1-cell in one 0-cell is $\pm 1$, see chapter II of \cite{W-18}.

Therefore chosen an orientation to each cell in a  decomposition cellular to the torus then the one-parameter trace $R(F)$ will be the form of the following matrix:
$$ R(F) = tr \left( \begin{array}{cc}
[-\tilde{\partial}_{1}] \otimes [\tilde{D}_{0}] & 0 \\
0 & [\tilde{\partial}_{2}] \otimes [\tilde{D}_{1}] \\
\end{array} \right )
$$
where $[\tilde{\partial_{1}}]_{ij}, [\tilde{\partial_{2}}]_{kl} \in \{0, \pm 1, \pm u,\pm v,  \}.$ Thus, we can write 
$$
R(F) = \pm 1 \otimes (\displaystyle\sum_{i=1}^{m} g_{i} ) +  u \otimes (\displaystyle\sum_{j=1}^{n} h_{j} ) + 
v \otimes (\displaystyle\sum_{k=1}^{p} t_{k} )
$$
or only $-u$ or $-v$, where $g_{i} = u^{m_{i}}v^{n_{i}}$, $h_{j} = u^{x_{j}}v^{y_{j}}$ and $t_{k} = u^{z_{k}}v^{w_{k}}$. We will suppose which $R(F)$ is write like above. The case with $-u$ or $-v$ the proof is analogous.  

From Lemma 4.1 of \cite{S-14} the element $\pm 1 \otimes (\displaystyle\sum_{i}^{m} g_{i} )$ is homologous to zero. By Proposition \ref{prop-h} we can suppose that $F$ has no fixed points in $T \times \{0,1\}$. In this situation $R(F)$ will be a 1-cycle in $HH_{1}(\mathbb{Z}G, (\mathbb{Z}G)^{\phi})$. Thus, By Proposition \ref{prop-2}, the sum  $v \otimes (\displaystyle\sum_{k}^{p} t_{k} )$ can not be appear in one-parameter trace $R(F)$ of $F$. Therefore, in this case the trace $R(F)$ has the form:
$$
R(F) = \pm 1 \otimes (\displaystyle\sum_{i=1}^{m} g_{i} ) +  u \otimes (\displaystyle\sum_{j=1}^{n} h_{j} ) $$

From Proposition \ref{prop-G} each C-component nonzero in $R(F)$ will represent by one unique cycle. Therefore the one-parameter Nielsen number in this case will be $N(F) = n$. 

From section 2 the one-parameter Lefschetz class is the image of $R(F)$ in $H_{1}(G)$ by induced of inclusion $i: Z(g_{C}) \to G$. Thus, each element $u \otimes h_{j}$ is sending in $H_{1}(G)$ in the class $[u]$, that is, the image of $R(F)$ in $H_{1}(G)$ will be 
$$ L(F) = \displaystyle\sum_{j=1}^{n} [u] = n[u] = N(F)[u]$$
Take $\alpha = [u]$, which is one of the two generators of $H_{1}(G)$. If we consider left action instead right action in the covering space  we will obtain  $L(F) = -N(F)[u]$.  Therefore, $$L(F) = \pm N(F)\alpha $$ 

\bigskip

{\bf Case II}

In this case we have $det([\phi]-I) = L(F_{|T}) \neq 0$. Therefore, by Corollary \ref{corol-2}, for each element $g \in G$ the semicentralizer, 
$Z(g)$, of $g$ in $G$ is trivial. Thus, $H_{1}(Z(g_{C})) = 0$ for each semiconjugacy class $C$, that is, 
$HH_{1}(\mathbb{Z}G, (\mathbb{Z}G)^{\phi}) = 0$ which implies $R(F) = 0.$ In this case we have $L(F) = N(F) = 0.$  \qed \end{proof}

\bigskip

We have other interpretation in Case II. Note that by definition of $R(F)$ in section 2 we are not considering in trace $R(F)$ the semiconjugacy classes represented by fixed point classes which meet $T \times \{0,1\}$. If we consider all fixed points classes then the trace $R(F)$ has the form:
$$ R(F) = \pm 1 \otimes (\displaystyle\sum_{i=1}^{m} g_{i} ) + v \otimes (\displaystyle\sum_{k=1}^{p} t_{k}  $$
because in this situation can not be appear circles in $Fix(F)$, but only arcs join $T \times \{0\}$ to $T \times \{ 1 \}$. 
By Proposition \ref{prop-2} $R(F)$ can not be a 1-cycle. Since for each $t$ the map $F_{t}$ can be deformed to a map with 
$L(F_{| T})$ fixed points, then from Theorem 3.3 of \cite{G-K-02} we will have $p = L(F_{|T}) = det([\phi]-I),$ i.e. in this 
case $Fix(F)$ will be compose by $det([\phi]-I)=L(F_{|T})$ arcs join $T \times \{ 0 \}$ to $T \times \{ 1 \}.$

\begin{figure}[!htp]
\begin{center}
		\includegraphics[scale=0.28]{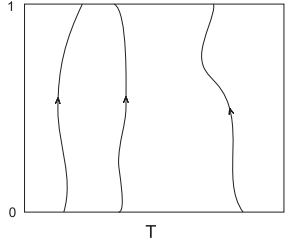}
	\caption{Arcs in $Fix(F).$}
	\end{center}
\end{figure}






\begin{thebibliography}{00}

\bibitem{B-B-P-T-75} {\it R.B.S.Brooks, R.F.Brown, J.Park and D.H.Taylor}, { Nielsen numbers of maps of tori}, 
Proc. Amer. Math. Soc., vol. 52, 1975. 

\bibitem{D-94} {\it D. Dimovski}, { One-parameter fixed point indices}, Pacific Journal of Math. 2, vol. 164, 1994.  

\bibitem{D-G-90} {\it D. Dimovski and R. Geoghegan}, {  One-parameter fixed point theory}, Forum Math. 2,
1990, 125-154.

\bibitem{G-K-02} {\it D.L.Gon\c calves, M.R.Kelly}, { Maps into the torus and minimal coincidence sets for homotopies}, 
Fund. Math., vol. 172, 2002.

\bibitem{G-N-94} { \it R. Geoghegan and A. Nicas}, { Parametrized Lefschetz-Nielsen fixed 
 point theory and Hochschild homology traces}, Amer. J. Math. 116, 1994, 397-446. 

\bibitem{G-N-94-2} { \it R. Geoghegan and A. Nicas}, { Trace and torsion in the theory of flows}, Topology, vol. 33, 
No. 4, 1994, 683-719.

\bibitem{G-N-S-00} { \it R. Geoghegan, A. Nicas and D. Sch$\ddot{u}$tz}, { Obstructions to homotopy invariance in
 parametrized fixed point theory}, Geometry and Topology: Aarhus, Contemp. Math. vol. 258, 2000, 351-369.

\bibitem{S-14} {\it W. L. Silva}, { Minimal fixed point set of fiber-preserving maps on T-bundles over $S^{1}$},
Topology and its Applications, 173, 2014, 240-263.

\bibitem{Sc-83} {\it H. Schirmer}, { Fixed points sets of homotopies},  Pacific J. Math.  108,  1983, 191--202. 

\bibitem{W-18} {\it G. W. Whitehead}, { Elements of Homotopy Theory}, Springer-Verlag, $1918.$




 \end{thebibliography}


\end{document}